\documentclass[12pt]{amsart}
\usepackage[all]{xy}
\usepackage{amssymb}
\calclayout
\newtheorem{dummy}{anything}[section] 

\newtheorem*{thma}{Theorem A}
\newtheorem*{thmb}{Theorem B}

\newtheorem*{question}{Question} 
\newtheorem{lemma}[dummy]{Lemma}

\theoremstyle{definition}%%Change Theoremstyle
\newtheorem{definition}[dummy]{Definition}

 \newtheorem{remark}[dummy]{Remark}
 \newtheorem*{acknowledgement}{Acknowledgement}

\newcommand
{\eqncount}{\setcounter{equation}{\value{dummy}}%
\addtocounter{dummy}{1}}
\newcommand{\cP}{\mathcal P}

\newcommand{\cE}{\mathcal E}

\newcommand{\bz}{\mathbf z}
\newcommand{\bw}{\mathbf w}

\newcommand{\bH}{\mathbf H}
\newcommand{\bZ}{\mathbf Z}
\newcommand{\bQ}{\mathbf Q}

\newcommand{\bC}{\mathbb C}
\newcommand{\bR}{\mathbb R}

\newcommand{\cy}[1]{\bZ/{#1}}

\newcommand{\wP}{\Gamma}
\newcommand{\bd}{\partial}
\newcommand{\vv}{\ | \ }

\newcommand{\ZE}{\bZ E}

\newcommand{\mmatrix}[4]{\left (\vcenter
{\xymatrix@C-2pc@R-2pc{#1&#2\\#3&#4} }
\right )}

\DeclareMathOperator{\wh}{Wh}
\DeclareMathOperator{\Mod}{mod}

\newcommand{\hsp}[1]{{\hphantom{#1}}}
 \newcommand{\la}{\langle}
  \newcommand{\ra}{\rangle}

\begin{document}
\title[Free Actions on Products of Spheres]
{Examples of Free Actions on Products of Spheres}
\author{Ian Hambleton}
\address{Department of Mathematics \& Statistics
\newline\indent
McMaster University
\newline\indent
Hamilton, ON L8S 4K1, Canada}
\email{ian@math.mcmaster.ca}
%information for second author
\author{\"Ozg\"un \"Unl\"u}
\address{Department of Mathematics \& Statistics
\newline\indent
McMaster University
\newline\indent
Hamilton, ON L8S 4K1, Canada}
\email{unluo@math.mcmaster.ca}
\date{May 21, 2008}
\thanks{\hskip -11pt  Research partially supported by NSERC Discovery Grant A4000. }

\begin{abstract}
We construct a  non-abelian extension  $\wP$ of $S^1$ by $\cy 3 \times \cy 3$, and prove that $\wP$ acts freely and smoothly on $S^5 \times S^5$. This gives new actions on $S^5 \times S^5$ for an infinite family $\cP$ of finite $3$-groups.  We also show that any finite odd order subgroup of the exceptional Lie group $G_2$ admits a free smooth action on $S^{11}\times S^{11}$.  This gives new  actions on $S^{11}\times S^{11}$ for an infinite family $\cE $ of finite groups.
We explain the significance of these families $\cP $, $\cE $ for the general existence problem, and correct some mistakes in the literature. 
\end{abstract}

\maketitle

\section*{Introduction}
In this paper we construct some new examples of smooth, free,  finite group actions on a product of two spheres of the same dimension.
A necessary condition discovered by Conner  \cite{conner1} is that $G$ has rank at most two, meaning that $G$ does not contain an elementary abelian subgroup of order $p^3$, for any prime $p$.

\begin{question}
What group theoretic conditions characterize the rank two finite groups which can act freely and smoothly on $S^n \times S^n$, for some $n\geqq 1$~?
\end{question}

It was shown by Oliver \cite{oliver1} that  the alternating group $A_4$ of order $12$  has rank two, but does not admit such an action, so the rank two condition is not sufficient. It was also observed by Adem-Smith \cite[p.~423]{adem-smith}  that $A_4$ is a subgroup of every rank two non-abelian simple group, so all these are ruled out too. 

In order to answer this question, it is useful to have more examples. In this note we present two new infinite families of such actions. Let $\wP$ be the Lie group given by the following presentation
\begin{equation*}
\wP= \left\langle a,b,z \vv z  \in S^1 \text{, }a^{3}=b^{3}=[a,z ]=[b,z ]=1\text{, }
[a,b]= \omega\right\rangle
\end{equation*} 
where $[x,y]=x^{-1}y^{-1}xy$ and $\omega=e^{2 \pi i/3 }$ in $S^1 \subseteq \mathbb{C}$. We make an explicit equivariant glueing construction to prove our first result.
\begin{thma}  The group $\wP$ acts freely and smoothly on $S^5\times S^5$.
\end{thma}
For a positive integer $k \geq 3$, let $P(k)$ be the group of order $3^{k}$ given by the following presentation
\begin{equation*}
P(k)= \left\langle a,b,c \vv a^{3}=b^{3}=c^{3^{k-2}}=[a,c]=[b,c]=1\text{, }
[a,b]= c^{3^{k-3}} \right\rangle
\end{equation*} 
We will write
$$\cP =\{P(k) \vv k\geq 3 \}$$
and note that $\cP$ is a collection of subgroups of $\wP$ (take $c = e^{2\pi i/3^{k-2}}\in S^1$). Therefore Theorem A constructs free smooth $P(k)$-actions on $S^5 \times S^5$ for all $k\geq 3$. Note that $P(3) \cong (\cy 3 \times \cy 3)\rtimes \cy 3$  is the extraspecial $3$-group of order $27$ and exponent $3$.

\medskip
We prove our second result by using equivariant surgery theory to modify a construction based on the exceptional Lie group $G_2$ of dimension $14$.  
\begin{thmb} All odd order finite subgroups of  $G_2$  act freely and smoothly on $S^{11}\times S^{11}$.
\end{thmb}

Information about the finite subgroups of $G_2$ can be found in  \cite{cohen-wales1}. Here is a specific family of examples.
For a prime number $p$, let $E(p)$ be the group of order $3p^2$ given by the following presentation
\begin{equation*}
E(p)=\left\langle u,v,w\vv  u^{p}=v^{p}=w^{3}=[u,v]=1\text{, }[u,w]=u^{-2}v^{-1}\text{, }[v,w]=uv^{-1} \right\rangle\ .
\end{equation*} 
We will write
$$\cE =\{E(p) \vv p\text{ is an odd prime}\}\ .$$
The group $E(2)$ is isomorphic to the alternating group $A_4$ of order $12$, and the group $E(3)$  is another presentation for the extraspecial group $P(3)$. An explicit isomorphism $P(3) \cong E(3)$ is given by the map
$$ a \mapsto w  \text{\ , \ \ \ } b \mapsto vu \text{\ , \ and \ \ } c \mapsto v^{-1}u \ .$$
The groups $E(p)$ are all subgroups of $SU(3)$, and hence contained in the exceptional Lie group $G_2$. For $p=3$, 
let $\omega = e^{2 \pi i/3 }$ and consider the representation of $P(3)$ as follows:
\begin{equation*}
a= \left[ \begin{array}{ccc}
0 & 1 & 0 \\ 
0 & 0 & 1 \\ 
1 & 0 & 0
\end{array}
\right] \text{, } b= \left[ 
\begin{array}{ccc}
1 & 0 & 0 \\ 
0 & \omega & 0 \\ 
0 & 0 & \omega ^{2}
\end{array}
\right] \text{, and } c= \left[ 
\begin{array}{ccc}
\omega & 0 & 0 \\ 
0 & \omega & 0 \\ 
0 & 0 & \omega
\end{array} \right]\ .
\end{equation*}
For $p \neq 3$, define $\alpha =e^{2\pi i/p}$ and $\beta =e^{2\pi i(p-2)/p}$
and consider a representation of $E(p)$ as follows: 
\[
u=\left[ 
\begin{array}{ccc}
\alpha  & 0 & 0 \\ 
0 & \alpha  & 0 \\ 
0 & 0 & \beta 
\end{array}
\right]  \text{, } v=\left[ 
\begin{array}{ccc}
\alpha  & 0 & 0 \\ 
0 & \beta  & 0 \\ 
0 & 0 & \alpha 
\end{array}
\right] \text{, and } w=\left[ 
\begin{array}{ccc}
0 & 1 & 0 \\ 
0 & 0 & 1 \\ 
1 & 0 & 0
\end{array}
\right] \ .
\]
Therefore, Theorem B proves the existence of free smooth $E(p)$-actions on $S^{11} \times S^{11}$, for all odd primes $p$.

\medskip
We introduce one more family of $3$-groups 
\begin{equation*}
B(k,\epsilon )= \left\langle a,b,c \vv a^{3}=b^{3}=c^{3^{k-2}}=[b,c]=1\text{, } [a,c]=b \text{, }
[a,b]= c^{\epsilon  3^{k-3}} \right\rangle
\end{equation*} 
where $k\geq 4$, and $\epsilon $ is $1$ or $-1$.
One can check that $B(k,\epsilon)$ is not a subgroup of $SU(3)$
for $k > 4$ or $\epsilon = 1$. However, 
the group $B(4,-1)$ is a subgroup of $SU(3)$, by the following representation
\begin{equation*}
a = \left[ \begin{array}{ccc}
0 & 1 & 0 \\ 
0 & 0 & 1 \\ 
1 & 0 & 0
\end{array}
\right] \text{, } b = \left[ 
\begin{array}{ccc}
1 & 0 & 0 \\ 
0 & \gamma ^{3} & 0 \\ 
0 & 0 & \gamma ^{6}
\end{array}
\right] \text{, } c= \left[ 
\begin{array}{ccc}
\gamma ^5 & 0 & 0 \\ 
0 & \gamma ^8 & 0 \\ 
0 & 0 & \gamma ^5
\end{array} \right]
\end{equation*}
where $\gamma =e^{2\pi i/9}$. Therefore Theorem B shows that $B(4,-1)$ acts freely and smoothly on $S^{11}\times S^{11}$. 

In Section 3 we make some concluding remarks about finite $3$-groups and the role of the families $\cP$ and $\cE$ in the general existence problem.

\begin{acknowledgement} The authors would like to thank Alejandro Adem, Dave Benson, Jim Davis and Matthias Kreck for useful conversations and correspondence.
\end{acknowledgement}

\section{An explicit construction}

The idea of the construction is to start with a non-free action of $\wP$ on $S^5\times S^5$ and do an equivariant ``cut-and-paste" operation on it to get rid of the fixed points. 
This is an equivariant surgery construction, but none of the theory of equivariant surgery is needed: the proof of Theorem A just involves checking some explicit formulas.

For the initial action  on $S^5\times S^5$, the singular set is contained in a $\wP$-invariant disjoint union $U$ of six copies of $S^1\times D^4 \times S^5$.  We replace this part by a new free action on $U$, which is  $\wP$-equivariantly diffeomorphic to the original one on its boundary.
We will use the following four representations of $\wP$ in our construction.
\begin{enumerate}\addtolength{\itemsep}{0.2\baselineskip}
\item An irreducible representation $\varphi \colon \wP\to U(3)$:
\begin{equation*}
a\longmapsto \left[ \begin{array}{ccc}
0 & 1 & 0 \\ 
0 & 0 & 1 \\ 
1 & 0 & 0
\end{array}
\right] \text{, }b\longmapsto \left[ 
\begin{array}{ccc}
1 & 0 & 0 \\ 
0 & \omega & 0 \\ 
0 & 0 & \omega ^{2}
\end{array}
\right] \text{, }z \longmapsto \left[ 
\begin{array}{ccc}
z  & 0 & 0 \\ 
0 & z  & 0 \\ 
0 & 0 & z 
\end{array}
\right]
\end{equation*}

\item Three representations that pullback from representations of $\wP/S^1$:
\smallskip
\begin{enumerate}\addtolength{\itemsep}{0.1\baselineskip}
\item $\psi _{0}\colon \wP\to U(3)$ given by:
\begin{equation*}
a\longmapsto \left[ 
\begin{array}{ccc}
\omega & 0 & 0 \\ 
0 & \omega & 0 \\ 
0 & 0 & \omega
\end{array}
\right] \text{, }b\longmapsto \left[ 
\begin{array}{ccc}
\omega & 0 & 0 \\ 
0 & \omega & 0 \\ 
0 & 0 & 1
\end{array}
\right] \text{, }z \longmapsto \left[ 
\begin{array}{ccc}
1 & 0 & 0 \\ 
0 & 1 & 0 \\ 
0 & 0 & 1
\end{array}
\right]
\end{equation*}

\item $\psi _{1}\colon \wP\to U(3)$ given by:
\begin{equation*}
a\longmapsto \left[ 
\begin{array}{ccc}
\omega & 0 & 0 \\ 
0 & \omega & 0 \\ 
0 & 0 & \omega
\end{array}
\right] \text{, }b\longmapsto \left[ 
\begin{array}{ccc}
\omega & 0 & 0 \\ 
0 & \omega ^{2} & 0 \\ 
0 & 0 & \omega ^{2}
\end{array}
\right] \text{, }z \longmapsto \left[ 
\begin{array}{ccc}
1 & 0 & 0 \\ 
0 & 1 & 0 \\ 
0 & 0 & 1
\end{array}
\right]
\end{equation*}

\item $\psi _{2}\colon \wP\to U(3)$ given by:
\begin{equation*}
a\longmapsto \left[ 
\begin{array}{ccc}
\omega & 0 & 0 \\ 
0 & \omega & 0 \\ 
0 & 0 & \omega
\end{array}
\right] \text{, }b\longmapsto \left[ 
\begin{array}{ccc}
\omega ^{2} & 0 & 0 \\ 
0 & 1 & 0 \\ 
0 & 0 & 1
\end{array}
\right] \text{, }z \longmapsto \left[ 
\begin{array}{ccc}
1 & 0 & 0 \\ 
0 & 1 & 0 \\ 
0 & 0 & 1
\end{array}
\right]
\end{equation*}
\end{enumerate}
\end{enumerate}
These representations give an action $\Phi \colon \wP\times Y\to Y$ on $Y=S^5$ 
given by:
$$\Phi (g,\bz) =\varphi (g)\bz,$$
where $\bz = (z_1, z_2, z_3)\in S^5$, with $z_i \in \bC$ and $\|\bz\| = 1$.
\begin{definition}[Model actions on $S^5\times S^5$]
For $i=0$, $1$, or $2$ we obtain an action $\Phi _{i}\colon \wP\times X_{i}\to X_{i}$ on $X_{i}=S^5\times 
S^5$ 
given by:
$$\Phi _{i}(g,(\bz,\bw)) =(\varphi (g)\bz,\psi _{i}(g)\bw),$$
where $\bz, \bw\in S^5$. 
\end{definition}
To simplify our notations, we let $\Phi (g,\bz)=g\cdot \bz$ and $\Phi
_{i}(g,(\bz,\bw))=g\cdot (\bz,\bw)$, for any $\bz\in Y$
and $(\bz,\bw)\in X_{i}$. 
\begin{remark}
We will modify the initial action $(X_0, \Phi_0)$ by ``equivariant Dehn surgery" to obtain a free $\wP$-action on $S^5\times S^5$, with replacement pieces coming from $(X_1, \Phi_1)$ and $(X_2, \Phi_2)$.
\end{remark}
For $i=0$, $1$, or $2$, we define a $\wP$-equivariant map 
\begin{equation*}
p_{i}\colon X_{i}\to Y\text{ given by }p_{i}(\bz,\bw)=\bz\ .
\end{equation*}
Note that $p_{i}$ is in fact a $\wP$-equivariant sphere bundle map.
Fix $0<\varepsilon <\frac{1}{9}$, and define three subspaces 
$V_{1}$, $V_{2}$, and $V_{0}$ of $Y$ as follows: 
$$V_1 = \{ a^k\cdot \bz \in Y \vv 0\leq k \leq 2, |z_2|^2 + |z_3|^2 \leq \varepsilon\}, \quad V_2 = PV_1$$
where 
$$ P = \frac{1}{\sqrt{3}}\left[\vcenter{\xymatrix@=2pt{1 & \omega  & 1 \\ 
1 & 1 & \omega  \\ 
\omega  & 1 & 1}}
\right] \ \in \ U(3)\ .$$
Note that $P\varphi (a)P^{-1}=\varphi (a)$ and $P\varphi (b)P^{-1}=\varphi
(a^{2}b)$, and let $V_{0}$ be the closure of $Y-V_{1}\cup V_{2}$.

\begin{lemma}
$V_{1}\cap V_{2}=\emptyset $.
\end{lemma}

\begin{proof}
Suppose $\bz\in V_{1}\cap V_{2}$. Then there exists $
\bz'\in V_{1}$ such that $
\bz=P\bz'$, since $\bz\in V_{2}$. So there exists $
i\neq j\in \{1,2,3\}$ such that $\left| z_{i}^{\prime }\right| ^{2}+\left|
z_{j}^{\prime }\right| ^{2}\leq \varepsilon $, since $\bz'\in V_{1}$. Let $
\{k\}=\{1,2,3\}-\{i,j\}$. Then for any $q$ in $\{1,2,3\}$ we have $\left|
z_{q}\right| ^{2}\geq \frac{1}{3}(\left| z_{k}^{\prime }\right| ^{2}-\left|
z_{i}^{\prime }\right| ^{2}-\left| z_{j}^{\prime }\right| ^{2})\geq
 \frac{1}{3}-\varepsilon $. Therefore any sum $|z_q|^2 + |z_r|^2 \geq \frac{2}{3} -2\varepsilon > \varepsilon$, in contradiction to the condition $\bz \in V_1$.
\end{proof}

\begin{lemma} The inclusions 
$t_{i}\colon V_{i}\to Y$ give $\wP$-equivariant subspaces of $Y$.
\end{lemma}
\begin{proof}
Assume $1\leq i\leq 2$. Take any $\bw $ in $V_{i}$, there exists unique $k\in \{0,1,2\}$ and $\bz$ in $V_1$ with $\left| z_{2}\right| ^{2}+\left| z_{3}\right|^{2} \leq \varepsilon $ such that
$$\bw = P^{i-1}\varphi (a^{k})\bz\ .$$
Hence $\varphi (a)\bw = P^{i-1}\varphi (a^{k+1})\bz$ is in $V_{i}$ and for $\lambda \in S^1$, 
$\varphi (\lambda) \bw = P^{i-1}\varphi (a^{k})\varphi (\lambda)\bz$ is in $V_{i}$ as $\left|\lambda z_{2}\right| ^{2}+\left|\lambda z_{3}\right|^{2} \leq \varepsilon $. We have 
\eqncount
\begin{equation}\label{formula}
\varphi (b) P^{i-1}\varphi (a^{k})= P^{i-1}\varphi (a^{-2(i-1)})\varphi (b)\varphi (a^{k})=P^{i-1}\varphi (a^{k+i-1})\varphi (b)\varphi(\omega^{-k})
\end{equation}
Hence for $i=1$, $\varphi (b)\bw = \varphi (a^{k})\varphi (b)\varphi (\omega^{-k})\bz$ is in $V_{i}$ as $\left|\omega ^{-k+1} z_{2}\right| ^{2}+\left|\omega ^{-k+2} z_{3}\right|^{2} \leq \varepsilon $. 
For $i=2$, $\varphi (b)\bw = P\varphi (a^{k+1})\varphi (b)\varphi (\omega^{-k})\bz$ is in $V_{i}$ as above. Hence the lemma is proved for $i=1$ and $i=2$. For $i=0$, it follows from the definition of $V_0$.
\end{proof}
\begin{remark} Observe that each of the subspaces $V_1$ or $V_2$ is  diffeomorphic to the disjoint union of three copies of $S^1 \times D^4$, since the subset $\{ \bz \in S^5 :  |z_2|^2 + |z_3|^2 \leq \varepsilon\} = S^1 \times D^4$.
\end{remark}
Now define a subpace $U_{i}\subset X_{i}$ for $i=0$, $1$, or $2$ by the
following $\wP$-equivariant pulback diagram:
$$\xymatrix@C+2pt{U_{i}\ar[r]\ar[d] &  X_{i}\ar[d]^{p_{i}}\cr
V_{i}\ar[r]^{t_{i}}&Y
}$$
\begin{lemma}
The $\wP$--action on $U_{i}$ is free for $i\in \{0,1,2\}$.
\end{lemma}

\begin{proof}
Take two subsets of $\wP$ as follows: 
\begin{eqnarray*}
A_{1} &=&\left\{ b^{k}z  \vv 1\leq k \leq 2, z  \in S^{1}\right\} \\
A_{2} &=&\left\{ a^{k}b^{-k}z  \vv 1\leq k \leq 2, z  \in S^{1}  \right\}
\end{eqnarray*}
All elements of $\wP$ except $A_{1}\cup A_{2}$ act freely on $X_{0}$. But all
the fixed point sets of elements of $A_{i}$ are in $p_{0}^{-1}(V_{i}-
\partial V_{i})$ for $i\in \{1,2\}$. Hence $\wP$ acts freely on $U_{0}$.
Now for any $i\in \{1,2\}$, all elements of $\wP$ except $A_{i}$ act freely
on $V_{i}$, but all the elements of $A_{i}$ act freely on $X_{i}$. Hence $\wP$
acts freely on $U_{i}$.
\end{proof}
\begin{remark} Since $U_i$ is an $S^5$-bundle over $V_i$,  the subspace $U = U_1\, \cup \,U_2$ is diffeomorphic to a disjoint union of six copies of $ S^1 \times D^4 \times S^5$. 
\end{remark}

\begin{lemma}\label{isom} There is a $\wP$-equivariant isomorphism $\alpha\colon \bd U_0 \to \bd U_1\, \cup\, \bd U_2$
as $\wP$-equivariant $5$-dimensional sphere bundles over $
\partial V_{0}=\partial V_{1}\cup \partial V_{2}$ with structure group $U(3)$.
\end{lemma}

\begin{proof}
For $m=1$ and $2$ we have:
$$\partial V_{m}=\{ P^{m-1}\varphi (a^{k})\bz \in Y\vv 0\leq k \leq 2, 
|z_2|^2 + |z_3|^2 = \varepsilon\}, $$
and 
$
\partial V_{0}=\partial V_{1}\cup \partial V_{2}
$.
This means that there is a unique way to write every element of $\partial U_{0}$
in the following standard form
$$( P^{m-1}\varphi (a^{k})\bz, \bw)$$
where $m\in \{1,2\}$, $k\in \{0,1,2\}$, and $\left| z_{2}\right| ^{2}+\left| z_{3}\right|
^{2}=\varepsilon $.
In addition,
$
\partial U_{n}=\partial V_{n}\times S^5
$, for $n=0$, $1$, and $2$, with $\wP$-action given by $g\cdot (\bz, \bw) = (\varphi(g)\bz, \psi_i(g)\bw)$.
We define an isomorphism 
$$
\alpha \colon \partial U_{0}\to \partial U_{1}\cup \partial U_{2}
$$
given by
$$
\alpha ( P^{m-1}\varphi (a^{k})\bz, \bw
) =\left( P^{m-1}\varphi (a^{k})\bz, \Theta_m(\bz)\bw
\right ),
$$
where
$$\Theta_1(\bz) = \frac{1}{\sqrt{\varepsilon (1-\varepsilon )}}\left [ \vcenter{\xymatrix@R=0pt@C=-2pt
{1&0&0\cr
0&\hsp{-}\bar z_1 z_2& -\bar{z}_1 z_3\cr
0 & \hsp{-} z_1 \bar{z}_3 & \hsp{-}z_1 \bar z_2
}} \right] \ \in \ SU(3)
$$
\medskip
$$\Theta_2(\bz) =\frac{1}{\sqrt{\varepsilon (1-\varepsilon )}}\left [ \vcenter{\xymatrix@R=0pt@C=-2pt
{\bar z_1 z_2& -z_1\bar z_3&0\cr
\bar z_1 z_3 & \hsp{-}z_1 \bar z_2& 0\cr
\hsp{-}0 &\hsp{-}0 & 1
}} \right] \ \in \ SU(3)
$$
Now it is clear that $\alpha $ is an isomorphism. We just have to check that
it is $\wP$-equivariant.

\medskip
\noindent
First,  check that $\alpha$ is equivariant under $a$:

\medskip
\noindent
\quad $\alpha \left ( a\cdot\left( P^{m-1}\varphi (a^{k})
\bz,\bw\right)\right ) =\alpha  \left( \varphi (a)P^{m-1}\varphi
(a^{k})\bz ,\psi _{0}(a)\bw\right)  $

\medskip
\noindent
$\hphantom{(\ast)}=\alpha  \left( P^{m-1}\varphi (a^{k+1})\bz ,\psi _{0}(a)\bw\right)  =\left( P^{m-1}\varphi (a^{k+1})\bz ,\Theta_{m}(\bz)\psi _{0}(a)\bw\right) $

\medskip
\noindent
$\hphantom{(\ast)}=\left( \varphi (a)P^{m-1}\varphi (a^{k})\bz ,\psi _{m}(a)\Theta_{m}(\bz)\bw \right) =a\cdot\alpha \left( P^{m-1}\varphi (a^{k})
\bz ,\bw \right) $

\medskip
\noindent
Second, check that $\alpha$ is equivariant under $b$:

\medskip
\noindent
\quad $\alpha \left( b\cdot\left( P^{m-1}\varphi (a^{k})\bz ,\bw\right) \right) 
=\alpha  \left( \varphi (b)P^{m-1}\varphi(a^{k})\bz ,\psi _{0}(b)\bw\right)  $

\medskip
\noindent
$\hphantom{(\ast)}=\alpha  \left( P^{m-1}\varphi (a^{k+m-1})\varphi (b)\varphi
(\omega ^{-k})\bz ,\psi _{0}(b)\bw\right) $, by formula  (\ref{formula}),

\medskip
\noindent
$\hphantom{(\ast)}=\alpha  \left( P^{m-1}\varphi (a^{k+m-1})\left[ 
\begin{array}{c}
\hsp{+1}\omega ^{-k}z_{1} \\ 
\omega ^{-k+1}z_{2} \\ 
\omega ^{-k+2}z_{3}
\end{array}
\right] ,\psi _{0}(b)\bw\right)  =(\star )$

\medskip
\noindent
For $m=1$ we have

\medskip
\noindent
$(\star )=\left( \varphi (a^{k})\left[ 
\begin{array}{c}
\hsp{+1}\omega ^{-k}z_{1} \\ 
\omega ^{-k+1}z_{2} \\ 
\omega ^{-k+2}z_{3}
\end{array}
\right] ,\frac{1}{\sqrt{\varepsilon (1-\varepsilon )}}\left[ 
\begin{array}{ccc}
1 & 0 & 0 \\ 
0 & \hsp{-}\bar{z}_{1}\omega z_{2} & -\bar{z}_{1}\omega ^2 z_{3} \\ 
0 & z_{1} \omega  \bar{z}_{3} & z_{1}\omega ^{2}\bar{z}_{2}
\end{array}
\right] \psi _{0}(b)\bw\right) $

\medskip
\noindent
$\hphantom{(\ast)}=\left( \varphi (b)\varphi (a^{k})\bz ,\Theta_{1}(\bz)\left[ 
\begin{array}{ccc}
1 & 0 & 0 \\ 
0 & \omega  & 0 \\ 
0 & 0 & \omega ^{2}
\end{array}
\right] \psi _{0}(b)\bw\right) =\left( \varphi (b)\varphi (a^{k})\bz ,\Theta_{1}(\bz)\psi _{1}(b)\bw\right) $

\medskip
\noindent
$\hphantom{(\ast)}=\left( \varphi (b)\varphi (a^{k})\bz ,\psi _{1}(b)\Theta_{1}(\bz)\bw \right) =b\cdot\alpha \left( \varphi (a^{k})\bz ,\bw\right) $

\medskip
\noindent
For $m=2$ we have

\medskip
\noindent
$(\star )=\left( P\varphi (a^{k+1})\left[ 
\begin{array}{c}
\hsp{+1}\omega ^{-k}z_{1} \\ 
\omega ^{-k+1}z_{2} \\ 
\omega ^{-k+2}z_{3}
\end{array}
\right] ,\frac{1}{\sqrt{\varepsilon (1-\varepsilon )}}\left[ 
\begin{array}{ccc}
\bar{z}_{1}\omega z_{2} & - z_{1}\omega \bar{z}_{3} & 0 \\ 
\bar{z}_{1}\omega ^{2}z_{3} & \hsp{-} z_{1}\omega ^{2}\bar{z}_{2} & 0 \\ 
0 & 0 & 1
\end{array}
\right] \psi _{0}(b)\bw\right)$ 

\medskip
\noindent
$\hphantom{(\ast)}=\left( \varphi (b) P \varphi (a^{k})\bz ,\left[ 
\begin{array}{ccc}
\omega  & 0 & 0 \\ 
0 & \omega ^{2} & 0 \\ 
0 & 0 & 1
\end{array}
\right] \Theta_{2}(\bz)\psi _{0}(b)\bw\right)$

\medskip
\noindent
$\hphantom{(\ast)}=\left( \varphi (b)P\varphi (a^{k})\bz ,\left[ 
\begin{array}{ccc}
\omega  & 0 & 0 \\ 
0 & \omega ^{2} & 0 \\ 
0 & 0 & 1
\end{array}
\right] \psi _{0}(b)\Theta_{2}(\bz)\bw \right)$

\medskip
\noindent
$\hphantom{(\ast)}=\left( \varphi (b)P\varphi (a^{k})\bz ,\psi _{2}(b)\Theta_{2}(\bz)\bw\right) =b\cdot\alpha\left( P\varphi (a^{k})\bz ,\bw\right) $

\medskip
\noindent
Third, check that $\alpha$ is equivariant under $\lambda \in S^1$:

\medskip
\noindent
\quad $\alpha \left( \lambda\cdot\left( P^{m-1}\varphi (a^{k})\bz ,\bw\right) \right) =\alpha \left( \varphi (\lambda)P^{m-1}\varphi
(a^{k})\bz ,\psi _{0}(\lambda)\bw\right) $

\medskip
\noindent
$\hphantom{(\ast)}=\alpha  \left( P^{m-1}\varphi (a^{k})\lambda\bz
,\bw\right)  =\left( P^{m-1}\varphi (a^{k})\lambda\bz,\Theta_{m}(\bz)\bw\right) $

\medskip
\noindent
$\hphantom{(\ast)}=\left( \varphi (\lambda)P^{m-1}\varphi (a^{k})\bz ,\psi _{m}(\lambda)\Theta_{m}(\bz)\bw\right) =\lambda\cdot\alpha \left( P^{m-1}\varphi (a^{k})
\bz ,\bw\right) $.
\end{proof}

\medskip
\begin{proof}[The proof of Theorem A]
Define a new space $X$ by the following pushout diagram
$$\xymatrix{
\partial U_{0}\cong \partial U_{1}\cup \partial U_{2}\ar[r]\ar[d]
& U_{1}\cup U_{2} \ar[d]\\ 
U_{0}\ar[r] & X
}$$
where the isomorphism $\alpha$ from Lemma \ref{isom} is used to make the identification $\bd U_0\cong \bd U_1\, \cup\, \bd U_2$. 
The above pushout diagram can be considered in the category of $\wP$-equivariant $5$-dimensional sphere bundles with the structure group $U(3)$.
Hence we see that $\wP$ acts freely on $X$ because the action of $\wP$ on $U_{1}\cup U_{2}$
and $U_{0}$ are both free. In addition, 
the base spaces of these bundles is given by the following
pushout diagram
$$\xymatrix{
\partial V_{0}=\partial V_{1}\cup \partial V_{2}\ar[r]\ar[d]& 
V_{1}\cup V_{2}\ar[d] \\ 
V_{0}\ar[r] & Y
}$$
Hence $X$ is a $5$-dimensional sphere bundle over $Y=S^5$
with structure group $U(3)$. But $\pi _{4}(U(3))=0$. Hence $X=S
^{5}\times S^5$.
\end{proof}

\section{Proof of Theorem B}

Let $E$ denote any finite odd order subgroup of the  exceptional Lie group $G_2 $.
To construct a free $E$-action on $S^{11}\times S^{11}$, we start with the 
free $E$-action on $G_2 $ given by left multiplication. Now consider the  principal
fibre bundle
$$S^3 =SU(2) \to G_2  \to G_2 /SU(2) =V_2(\bR^7)$$
with structure group $SU(2)$ over the Stiefel manifold $V_2(\bR^7)$.  This fibre bundle can be identified with the sphere bundle of an associated $2$-dimensional complex vector bundle $\xi$. By construction, the space 
$$Z(\xi) = G_2  \times_{SU(2)} \bC^2 $$ is the total space of the vector bundle $\xi$, where $SU(2)$ acts on $\bC^2$ via the standard representation, and freely off the zero-section.
It follows that the group $G_2$ acts on $Z(\xi)$ through left multiplication, and freely off the zero section.
We therefore obtain a free smooth $G_2$-action on 
the total space $Y$ of the sphere bundle 
$$S^{11} \to Y \to V_2(\bR^7)$$
of the complex vector bundle $\xi\oplus \xi\oplus \xi$. This action can be restricted to any finite subgroup of $G_2$, but the equivariant surgery construction given below to obtain a free action on $S^{11}\times S^{11}$ is valid only for the odd order subgroups $E$ of $G_2$.

\begin{lemma} \label{product_lemma}
$Y$  is a smooth, closed, parallelisable manifold diffeomorphic to $S^{11}\times V_2(\bR^7)$.
\end{lemma}
\begin{proof}
Consider the fibre bundle
$$SU(3)/SU(2) \to G_2 /SU(2)   \to G_2 /SU(3) $$
which is equivalent to $$S^5  \to V_2(\bR^7)   \to S^6\ .$$
By \cite[Prop.~7.5]{borel-hirzebruch1}, the tangent bundle along the fibers of the total space 
$V_2(\bR^7)$ is equivalent to $\xi $ after adding a trivial line bundle. It is known that the total space $V_2(\bR^7)$ is parallelisable (see 
\cite[Corollary]{bredon-kosinski1}), and the tangent bundle of the base $S^6$ is stably trivial. Therefore $\xi$ is stably trivial over
$V_2(\bR^7)$, which means that the $12$-plane bundle $\xi\oplus \xi\oplus \xi$ is trivial over
$V_2(\bR^7)$ as the dimension of $V_2(\bR^7)$ is $11$. This proves $Y$ is
diffeomorphic to $S^{11} \times V_2(\bR^7)$. We also know that the tangent bundle of $S^{11}$ is stably trivial, hence $Y$ parallelisable.
 \end{proof}

\begin{lemma} \label{homology_lemma}
$Y$ is $4$-connected and has the integral homology of $S^{11}\times
S^{11}$, except for the groups
$H_5(Y;\bZ) = H_{16}(Y; \bZ) = \cy 2$.
\end{lemma}
\begin{proof}
The proof is easy using Lemma \ref{product_lemma} and the fact that
$V_{2}(\mathbb{R}^{7})$ is $4$-connected, with integral homology given as
follows
$$
H_{q}(V_{2}(\mathbb{R}^{7}))=
\left\{
\vcenter{\xymatrix@R=-3pt{\bZ & \text{if\ } q=0\text{ or }q=11\cr
\hsp{xi}\cy 2 & \text{if \ }q=5\hsp{xxxxxxx}\cr
0 & \text{otherwise}\hsp{xxxxxx}\hfill
}}
\right\}\ .
$$
\vskip -14pt \qedhere
\end{proof}
We will now show how to perform 
$E$-equivariant framed surgery on $Y$ to obtain a free $E$-action on $S^{11}\times S^{11}$. In the successive steps, we remove the interior of an equivariant framed embedding of
$E \times S^k \times D^{22-k}$ and attach $E \times D^{k+1} \times S^{22-k-1}$ along their common boundaries. 

This is an equivariant version of the original spherical modification construction of Milnor \cite{milnor4}, \cite{kervaire-milnor1} which formed the starting point for surgery theory as developed by Browder, Novikov, Sullivan and Wall (see \cite{wall-book}. or the short overview in \cite[\S 7]{h-trieste}). We remark that non-simply connected surgery is carried out equivariantly in the universal covering of a manifold, where the equivariance is with respect to the action of the fundamental group as deck transformations.

In order to carry out $E$-equivariant framed surgery on $Y$,  we will need a partial
equivariant trivialization of the normal bundle of $Y$ to produce the framings. Let $X = Y/E$ and $\nu_X$ be the
classifying map  of the stable normal bundle of $X$. Since $Y$ is $4$-connected by Lemma \ref{homology_lemma},  we can
construct the classifying space $BE$ by adding $k$-cells to $X$ for $k>5$. 
Let $B= BE^{(12)}\cup X$, where $BE^{(12)}$ denotes the $12$-skeleton of $BE$,  . We have a pullback diagram
$$\xymatrix{
\ \ Y\ \ar@{^{(}->}@<-.5ex>[r]\ar[d] & \widetilde{B}\ar[d] \\ 
\ X\ \ar@{^{(}->}@<-.4ex>[r] & B }$$
of universal coverings.
 The assumption that $E$ has odd order will now be used for the first time.
\begin{lemma} \label{bundle_lemma} Since $E$ has odd order,
the normal bundle $\nu_X\colon X\to BSO$  is the restriction of a bundle $\nu \colon B\to BSO$.
\end{lemma}
\begin{proof}
The successive obstructions to extending the classifying map
$\nu_X\colon X\to BSO$ of the stable normal bundle of $X$ to a map from
$B$ to $BSO$ lie in the  groups
$$H^{k}(B,X;\pi_{k-1}(BSO))$$
for $k \geq 6$. We claim that these obstructions vanish since $E$ has odd order. For $6 \leq k \leq 7$, we have $\pi_{k-1}(BSO))=0$.
For $8 \leq k \leq 11$,  by considering Lemma \ref{homology_lemma} and the
cohomology long exact sequence of the pair $(B,X)$ with coefficients in
any abelian group $A$, we get $H^{k}(B,X;A)=0$. Finally for $k=12$, we
have $\pi_{11}(BSO) = 0$, so we may extend $\nu_X$ over  $B$.
\end{proof}
Let $B' = BE^{(11)}\cup X \subseteq B$,  and still denote the restriction of $\nu$ to $B'$ by $\nu $.
\begin{lemma} \label{bundle_lemma_two}
The pullback $\tilde \nu $ of $\nu$ by the map $\widetilde{B'}\to B'$ is trivial.
\end{lemma}
\begin{proof}
The normal bundle $\nu_{Y}$ of $Y$ is trivial, hence it is enough to extend a null homotopy of the map $\nu_{Y}$ to a null homotpy of $\tilde\nu $.
The successive obstructions for this extension problem lie in the groups
$$H^{k}(\widetilde{B'},Y;\pi_{k}(BSO))$$
for $k \geq 6$. We claim that these obstructions also vanish. For $6 \leq k \leq 7$, we have $\pi_{k}(BSO))=0$.
For $8 \leq k \leq 10$,  by considering Lemma \ref{homology_lemma} and the
cohomology long exact sequence of the pair $(\widetilde{B'},Y)$ with coefficients in
any abelian group $A$, we get $H^{k}(\widetilde{B'},Y;A)=0$. Since $\pi_{11}(BSO) = 0$ we are done.
\end{proof}

Let $\bH(L)$ denote the standard skew-hermitian hyperbolic form on the
module $L\oplus L^*$. The following uses surgery below the middle dimension,  a standard procedure in surgery theory 
(see\cite[\S 5]{kervaire-milnor1}, \cite[Chap.~1]{wall-book}).

\begin{lemma}\label{middle} After preliminary surgeries
on $X$, we can obtain a smooth manifold $M$ with the
following properties:
\begin{enumerate}\addtolength{\itemsep}{0.2\baselineskip}
\item $\widetilde{M}$ is stably parallelisable.
\item The classifying map $c\colon M \to BE$ induces an isomorphism $\pi_1(M) \cong E$.
\item $\pi_i(M) = 0$ for $1 < i < 11$.
\item The intersection form $$(\pi_{11}(M), s_M) \cong \bH(\bZ)\perp
(F,\lambda)$$ for some
non-singular skew hermitian form $\lambda$ on a finitely-generated free
$\ZE$-module $F$. 
\end{enumerate}
\end{lemma}
\begin{proof}
Lemma \ref{bundle_lemma} gives a bundle $\nu\colon B' \to BSO$. We will perform a sequence of surgeries over $(B', \nu)$, so that in particular the bundle $\nu$ pulls back to the stable normal bundle of the trace of the surgeries.  By Lemma  \ref{bundle_lemma_two}, the resulting manifold $M$ at any stage of these surgeries has universal covering $\widetilde M$ stably parallelisable.

The first step is surgery to kill a generator of $\pi_5(X)=\cy 2$. 
We use the short exact sequence
$$ 0 \to \la 2, I\ra \to \ZE \to \cy 2\to 0$$
of $\ZE$-modules, where $I$ denotes the augmentation ideal of $\ZE$, to
keep track of the effect of the first step of the $E$-equivariant framed
surgery on $Y$. The result of the first step is a manifold $M$ such that
$\pi_6(M)=\la 2, I\ra$.
We have a short exact sequence
$$0 \to \ZE \to \la 2, N\ra \to \cy 2\to 0,$$
where the module $ \la 2, N\ra$ is projective over $\ZE$ since $E$ has odd order 
(see \cite[\S 6]{swan1}).
Now Schanuel's Lemma shows that  
$$\la 2, N\ra \oplus \la 2, I\ra=\bZ E
\oplus \bZ E$$ is free over $\ZE$, so $\la 2, I\ra$ is a
finitely-generated projective $\ZE$-module with stable inverse $\la 2,
N\ra$. The effect of the subsequent  surgeries  to make $\widetilde M$
highly-connected is just to replace a projective module $\pi_i(M)=Q$ at
each step with its stable inverse $\pi_{i+1}(M')=Q'$, for $i <10$. At the last of these
steps, where we eliminate $\pi_{10}(M)$, the result is an expression
$$(\pi_{11}(M), s_M) \cong \bH(\bZ)\perp  (P, \lambda ')$$
where $(P,\lambda ')$ is a non-singular skew-hermitian form on $P = Q \oplus Q^*$, and $Q \cong\la 2, N\ra$.
The projective
modules $\la r, N\ra$, for $r$ prime to $|E|$, generate the Swan subgroup
$T(\ZE) \subseteq \widetilde K_0(\ZE)$ of the projective class group. Now
Swan \cite[Lemma 6.1]{swan1} proved that
$$\bZ \oplus \langle r, N\rangle \cong \bZ \oplus \bZ E$$
for any $r$ prime to $|E|$, and that 
$$\la 2, N\ra \oplus \la r, N \ra \cong \bZ E \oplus \bZ E$$
if  $2r\equiv 1\ (\Mod  |E|)$. 
After surgery on a null-homotopic $10$-sphere in $M$, we obtain $M'
=M\# (S^{11} \times S^{11})$, whose equivariant intersection form is
$$(\pi_{11}(M'), s_{M'}) \cong \bH(\bZ)\perp (P,\lambda ') \perp
\bH(\ZE)$$
However note that
$$\bH(\bZ)\perp \bH(\ZE) = \bH(\bZ\oplus \ZE) \cong \bH(\bZ\oplus \la
r,N \ra) = \bH(\bZ)\perp \bH(\la r,N \ra)\ .$$
Now $(F,\lambda) :=\bH(\la r,N \ra)\perp (P, \lambda ')$ is
a non-singular skew-hermitian form on a finitely-generated free
$\ZE$-module. 
\end{proof}
We next observe that the equivariant intersection form $(\pi_{11}(M),
s_M)$ has a quadratic refinement $\mu\colon \pi_{11}(M) \to \ZE/\{\nu +
\bar\nu\}$, in the sense of \cite[Theorem 5.2]{wall-book}. Since $E$ has
odd order, this follows because the universal covering $\widetilde M$ has
stably trivial normal bundle. We therefore obtain an element
$(F, \lambda, \mu)$
of the surgery obstruction group (see \cite[p.~49]{wall-book} for the
essential definitions). In the splitting $(\pi_{11}(M), s_{M}, \mu) = H(\bZ) \perp (F, \lambda, \mu)$ we may assume that the Arf invariant of the summand $H(\bZ)$ is zero. This follows by construction, since the preliminary surgeries can be done away from an embedded sphere $$S^{11} \times \ast \ \subset \ S^{11} \times V_2(\bR^7) =Y$$
with trivial normal bundle.
We need to check the discriminant of the form $(F, \lambda, \mu)$.
\begin{lemma}
We obtain an element
$$(F, \lambda, \mu) \in L'_{2}(\ZE)$$
of the weakly-simple surgery obstruction group.
\end{lemma}
\begin{proof}
A non-singular, skew-hermitian quadratic form $(F, \lambda, \mu)$
represents an element in
$L'_{2}(\ZE)$ provided that its discriminant lies in $\ker(\wh(\ZE) \to
\wh(\bQ E))$. But the equivariant symmetric Poincar\'e chain complex
$(C(M), \varphi_0)$ is chain equivalent, after tensoring with the
rationals $\bQ$, to its rational homology complex (see \cite[\S 4]{ra10}).
Therefore the image of the discriminant of $(\pi_{11}(M)\otimes \bQ, s_M)$
equals the image of the torsion of $\varphi_0$, which vanishes in $\wh(\bQ
E)$ because closed manifolds have simple Poincar\'e duality (see
\cite[Theorem 2.1]{wall-book}).
\end{proof}

\begin{proof}[The proof of Theorem B]
We now have a smooth closed manifold $[M,c]$ whose equivariant
intersection form $(\pi_{11}(M), s_M)$ contains $(F, \lambda, \mu)$, as
described above. However, since $E$ has odd order, an element in the surgery obstruction group
${L'}_{2}(\ZE)$
is zero provided that its multisignature and ordinary Arf invariant both
vanish (this is a result of Bak and Wall, see \cite[Cor.~2.4.3]{wall-VI}).
The multisignature invariant is trivial since $M$ is a closed manifold
\cite[13B]{wall-book}.  The ordinary Arf invariant of $(F, \lambda, \mu)$ equals the Arf invariant of $\widetilde M$, which vanishes since $22$ is not of the form $2^k -2$ (a
famous result of Browder \cite{browder1}). We can now do surgery  to obtain a representative $[M,c]$ which has
$\widetilde M = S^{11}\times S^{11} \# \Sigma$, where $\Sigma$ is a
homotopy $22$-sphere. Note that the $p$-component of $\pi_{22}^S$ is zero
for $p \geq 3$ (see  \cite[p.~5]{ravenel1}), so we can get the standard
smooth structure on $S^{11}\times S^{11}$.
\end{proof}

\section{Concluding Remarks}
In this final section we will make some additional remarks about the group theory, and explain the significance of constructing actions for our families $\cP$ and $\cE$ of finite groups, as a step towards answering our original question.

\smallskip
\noindent
\textbf{(I).} Blackburn has given a classification of $p$-groups of rank $2$. Here we restate his result for $3$-groups (see Theorem  4.1 in \cite{blackburn2} and Theorem 3.1 in \cite{leary3}). If $G$ is a rank two $3$-group of order $3^k$ then one of the following holds 
\begin{enumerate}
\item $G$ is a metacyclic $3$-group,
\item $G=P(k)$, $k \geq 3$, a group in $\cP$,
\item $G=B(k, \epsilon)$, $k \geq 4$, 
\item $G$ is a $3$-group of maximal class. 
\end{enumerate} 
The $3$-groups listed in the first item all act freely and smoothly on a product of two equidimensional spheres (see \cite[p.~538]{lewis1}). An explicit construction and the proof of Theorem A shows that the groups in the second item on this list  act freely on $S^5\times S^5$. Theorem B shows that  the group $B(4,-1)$ in the third item also acts freely on a product of two equidimensional spheres, but of dimension $S^{11} \times S^{11}$.

\smallskip
\noindent
\textbf{(II).} It was shown by Benson and Carlson (see Theorem 4.4 in \cite{benson-carlson1}) that
free actions of a rank two group on a product of two equidimensional spheres could not be ruled out by cohomological methods alone. Hence the arguments given for certain non-existence claims in \cite{alzubaidy1}, \cite{alzubaidy2}, \cite{thomas1}, and \cite{yagita1} about extraspecial $p$-groups are not valid. In fact, 
Theorems A and B  applied to the extraspecial $3$-group $E(3)$ of order $27$ and exponent $3$ give specific counterexamples to the results claimed in these papers. The possible sphere dimensions for this group $E(3)$, not previously ruled out by cohomological methods, are of the form $S^{6r-1} \times S^{6r-1}$, and our examples show existence in the first two cases ($r=1, 2$).

\smallskip
\noindent
\textbf{(III).} For any prime number $p$, the group $E(p)$ is a subgroup of $G_2$, but $E(2) \cong A_4$. Since $A_4$ is ruled out by \cite{oliver1}, Theorem B shows  that the group $E(p)$ can act freely and smoothly on a product of two equidimensional spheres if and only if $p>2$.  More information about the odd order subgroups of $G_2$ can be found in  \cite{cohen-wales1} (the finite subgroups are not all contained in $SU(3)$, but we don't know  if this is true for the odd order subgroups). The result of Oliver \cite{oliver1}  was also proved and extended by Carlsson \cite{carlsson1} and Silverman \cite{silverman1}.

\smallskip
\noindent
\textbf{(IV).} Let $G$ be a group in  $\cP$ or  $\cE$. Let $axe(G)$ be the minimum number of linear representations of $G$ required for $G$ to act freely on a product of spheres where the action on each sphere is induced from one of these representations. By \cite[Proposition 3.3]{barker-yalcin1}, it easy to see that $axe(G) \geq 3$. Hence $G$ can not act freely on a product of two sphere with a linear action on each sphere. Moreover $G$ is not a subgroup of $Sp(2)$, hence the free actions constructed in \cite{adem-davis-unlu} will not be on a product of two equidimensional spheres. We also remark that $G$ can not be written as a product of two groups with periodic cohomology, while all the subgroups of $G$ can. So the families $\cP$ and $\cE $ are two infinite families of minimal new examples not included in \cite{h4}.

%%%%%%%%%%%%%%%%%%%%%%%%%%%%%%%%%%%%%%%
\providecommand{\bysame}{\leavevmode\hbox to3em{\hrulefill}\thinspace}
\providecommand{\MR}{\relax\ifhmode\unskip\space\fi MR }
\providecommand{\MRhref}[2]{%
  \href{http://www.ams.org/mathscinet-getitem?mr=#1}{#2}
}
\providecommand{\href}[2]{#2}


\begin{thebibliography}{10}

\bibitem{adem-davis-unlu}
A.~Adem, J.~F. Davis, and {\"O}.~{\"U}nl{\"u}, \emph{Fixity and free group
  actions on products of spheres}, Comment. Math. Helv. \textbf{79} (2004),
  758--778.

\bibitem{adem-smith}
A.~Adem and J.~H. Smith, \emph{Periodic complexes and group actions}, Ann. of
  Math. (2) \textbf{154} (2001), 407--435.

\bibitem{alzubaidy1}
K.~Alzubaidy, \emph{Free actions of {$p$}-groups {$(p>3)$} on {$S\sp{n}\times
  S\sp{n}$}}, Glasgow Math. J. \textbf{23} (1982), 97--101.

\bibitem{alzubaidy2}
\bysame, \emph{Free actions on {$(S\sp n)\sp k$}}, Mathematika \textbf{32}
  (1985), 49--54.

\bibitem{barker-yalcin1}
L.~Barker and E.~Yal{\c{c}}{\i}n, \emph{A new notion of rank for finite
  supersolvable groups and free linear actions on products of spheres}, J.
  Group Theory \textbf{6} (2003), 347--364.

\bibitem{benson-carlson1}
D.~J. Benson and J.~F. Carlson, \emph{Complexity and multiple complexes}, Math.
  Z. \textbf{195} (1987), 221--238.

\bibitem{blackburn2}
N.~Blackburn, \emph{Generalizations of certain elementary theorems on
  {$p$}-groups}, Proc. London Math. Soc. (3) \textbf{11} (1961), 1--22.

\bibitem{borel-hirzebruch1}
A.~Borel and F.~Hirzebruch, \emph{Characteristic classes and homogeneous
  spaces. {I}}, Amer. J. Math. \textbf{80} (1958), 458--538.

\bibitem{bredon-kosinski1}
G.~E. Bredon and A.~Kosi{\'n}ski, \emph{Vector fields on {$\pi $}-manifolds},
  Ann. of Math. (2) \textbf{84} (1966), 85--90.

\bibitem{browder1}
W.~Browder, \emph{The {K}ervaire invariant of framed manifolds and its
  generalization}, Ann. of Math. (2) \textbf{90} (1969), 157--186.

\bibitem{carlsson1}
G.~Carlsson, \emph{On the rank of abelian groups acting freely on
  {$(S\sp{n})\sp{k}$}}, Invent. Math. \textbf{69} (1982), 393--400.

\bibitem{cohen-wales1}
A.~M. Cohen and D.~B. Wales, \emph{Finite subgroups of {$G\sb{2}({\bf C})$}},
  Comm. Algebra \textbf{11} (1983), 441--459.

\bibitem{conner1}
P.~E. Conner, \emph{On the action of a finite group on {$S\sp{n}\times
  S\sp{n}$}}, Ann. of Math. (2) \textbf{66} (1957), 586--588.

\bibitem{h-trieste}
I.~Hambleton, \emph{Algebraic {$K$}- and {$L$}-theory and applications to the
  topology of manifolds}, Topology of high-dimensional manifolds, No. 1, 2
  (Trieste, 2001), ICTP Lect. Notes, vol.~9, Abdus Salam Int. Cent. Theoret.
  Phys., Trieste, 2002, pp.~299--369.

\bibitem{h4}
\bysame, \emph{Some examples of free actions on products of spheres}, Topology
  \textbf{45} (2006), 735--749.

\bibitem{kervaire-milnor1}
M.~A. Kervaire and J.~W. Milnor, \emph{Groups of homotopy spheres. {I}}, Ann.
  of Math. (2) \textbf{77} (1963), 504--537.

\bibitem{leary3}
I.~J. Leary, \emph{The cohomology of certain groups}, Ph.D. thesis, University
  of Cambridge, 1991,
  (http://www.maths.abdn.ac.uk/~bensondj/html/archive/leary.html).

\bibitem{lewis1}
G.~Lewis, \emph{The integral cohomology rings of groups of order {$p\sp{3}$}},
  Trans. Amer. Math. Soc. \textbf{132} (1968), 501--529.

\bibitem{milnor4}
J.~Milnor, \emph{A procedure for killing homotopy groups of differentiable
  manifolds.}, Proc. Sympos. Pure Math., Vol. III, American Mathematical
  Society, Providence, R.I, 1961, pp.~39--55.

\bibitem{oliver1}
R.~Oliver, \emph{Free compact group actions on products of spheres}, Algebraic
  topology, Aarhus 1978 (Proc. Sympos., Univ. Aarhus, Aarhus, 1978), Lecture
  Notes in Math., vol. 763, Springer, Berlin, 1979, pp.~539--548.

\bibitem{ra10}
A.~Ranicki, \emph{The algebraic theory of surgery. {I}. {F}oundations}, Proc.
  London Math. Soc. (3) \textbf{40} (1980), 87--192.

\bibitem{ravenel1}
D.~C. Ravenel, \emph{Complex cobordism and stable homotopy groups of spheres},
  AMS Chelsea Publishing, 2004.

\bibitem{silverman1}
J.~H. Silverman, \emph{Subgroup conditions for groups acting freely on products
  of spheres}, Trans. Amer. Math. Soc. \textbf{334} (1992), 153--181.

\bibitem{swan1}
R.~G. Swan, \emph{Periodic resolutions for finite groups}, Ann. of Math. (2)
  \textbf{72} (1960), 267--291.

\bibitem{thomas1}
C.~B. Thomas, \emph{Free actions by {$p$}-groups on products of spheres and
  {Y}agita's invariant {$po(G)$}}, Transformation groups (Osaka, 1987), Lecture
  Notes in Math., vol. 1375, Springer, Berlin, 1989, pp.~326--338.

\bibitem{wall-VI}
C.~T.~C. Wall, \emph{Classification of {H}ermitian {F}orms. {V}{I}. {G}roup
  rings}, Ann. of Math. (2) \textbf{103} (1976), 1--80.

\bibitem{wall-book}
\bysame, \emph{Surgery on compact manifolds}, second ed., American Mathematical
  Society, Providence, RI, 1999, Edited and with a foreword by A. A. Ranicki.

\bibitem{yagita1}
N.~Yagita, \emph{On the dimension of spheres whose product admits a free action
  by a nonabelian group}, Quart. J. Math. Oxford Ser. (2) \textbf{36} (1985),
  117--127.

\end{thebibliography}
\end{document}